\documentclass{amsproc}

\usepackage{enumerate}

\usepackage{amssymb,caption,epigraph,graphicx,mathabx,mathrsfs,mathtools,psfrag,relsize,tikz,upgreek,xcolor,xparse}

\usetikzlibrary{calc,decorations.pathmorphing,shapes}

\definecolor{refkey}{rgb}{1,.5,.5}
\definecolor{labelkey}{rgb}{1,.5,.5}

\usepackage[OT2,OT1]{fontenc}

\usepackage{scalerel}

\usepackage{stackengine}
\stackMath

\DeclareRobustCommand\cyr{%
  \renewcommand\rmdefault{wncyr}%
  \renewcommand\sfdefault{wncyss}%
  \renewcommand\encodingdefault{OT2}%
  \normalfont
  \selectfont}

\DeclareTextFontCommand{\textcyr}{\cyr}

\DeclareSymbolFont{euleroperators}{U}{eur}{m}{n}
\SetSymbolFont{euleroperators}{bold}{U}{eur}{b}{n}

\makeatletter
\renewcommand{\operator@font}{\mathgroup\symeuleroperators}
\makeatother

\usepackage{fancyvrb}

\makeatletter
\newcommand{\raisemath}[1]{\mathpalette{\raisem@th{#1}}}
\newcommand{\raisem@th}[3]{\raisebox{#1}{$#2#3$}}
\makeatother

\makeatother

\makeatletter
\let\c@equation\c@figure
\makeatother

\makeatletter
\def\@seccntformat#1{%
  \protect\textup{\protect\@secnumfont
    \ifnum\pdfstrcmp{subsection}{#1}=0 \bfseries\fi
    \csname the#1\endcsname
    \protect\@secnumpunct
  }%
}
\makeatother

\makeatletter
\def\@tocline#1#2#3#4#5#6#7{\relax
   \ifnum #1>\c@tocdepth 
   \else
     \par \addpenalty\@secpenalty\addvspace{#2}%
     \begingroup \hyphenpenalty\@M
     \@ifempty{#4}{%
       \@tempdima\csname r@tocindent\number#1\endcsname\relax
     }{%
       \@tempdima#4\relax
     }%
     \parindent\z@ \leftskip#3\relax \advance\leftskip\@tempdima\relax
     \rightskip\@pnumwidth plus4em \parfillskip-\@pnumwidth
     #5\leavevmode\hskip-\@tempdima #6\nobreak\relax
     \ifnum#1<0\hfill\else\dotfill\fi\hbox to\@pnumwidth{\@tocpagenum{#7}}\par
     \nobreak
     \endgroup
   \fi}
\makeatother

\NewDocumentCommand{\bn}{sO{}m}{%
  {\IfBooleanTF{#1}
    {\oldnormaux{\left|}{\right|}{#3}}
    {\oldnormaux{#2|}{#2|}{#3}}}
}
\makeatletter
\newcommand{\oldnormaux}[3]{\mathpalette\oldnormaux@i{{#1}{#2}{#3}}}
\newcommand{\oldnormaux@i}[2]{\oldnormaux@ii#1#2}
\newcommand{\oldnormaux@ii}[4]{%
  \sbox\z@{$\m@th#1#2#4#3$}%
  \sbox\tw@{$\m@th\|$}%
  \mathopen{\hbox to\wd\tw@{\hss\vrule height \ht\z@ depth \dp\z@ width .2\wd\tw@\hss}}%
  #4
  \mathclose{\hbox to\wd\tw@{\hss\vrule height \ht\z@ depth \dp\z@ width .2\wd\tw@\hss}}%
}
\makeatother

\newtheorem{thm}[equation]{Theorem}

\newtheorem{lem}[equation]{Lemma}

\newtheorem{prp}[equation]{Proposition}

\newcommand{\thistheoremname}{}

\newtheorem*{genericthm*}{\thistheoremname}
\newenvironment{namedthm*}[1]
  {\renewcommand{\thistheoremname}{#1}%
   \begin{genericthm*}}
  {\end{genericthm*}}

\theoremstyle{definition}
\newtheorem{dfn}[equation]{Definition}

\theoremstyle{remark}
\newtheorem{rem}[equation]{Remark}

\newcommand{\thmref}[1]{Theorem~\ref{#1}}
\newcommand{\prpref}[1]{Proposition~\ref{#1}}
\newcommand{\lemref}[1]{Lemma~\ref{#1}}

\newcommand{\dfnref}[1]{Definition~\ref{#1}}
\newcommand{\remref}[1]{Remark~\ref{#1}}

\newcommand{\figref}[1]{Figure~\ref{#1}}
\newcommand{\secref}[1]{Section~\ref{#1}}

\DeclareRobustCommand\cyr{%
  \renewcommand\rmdefault{wncyr}%
  \renewcommand\sfdefault{wncyss}%
  \renewcommand\encodingdefault{OT2}%
  \normalfont
  \selectfont}

\newcommand{\conv}[1]
{\vskip -.3cm $$ {\vbox{\vrule width 1mm \hskip 4mm \hbox{\parbox{10cm}{ {#1}
}}}\hfill} $$}

\newcounter{sarrow}
\newcommand\xrsquigarrow[1]{%
\stepcounter{sarrow}%
\mathrel{\begin{tikzpicture}[baseline= {( $ (current bounding box.south) + (0,-0.5ex) $ )}]
\node[inner sep=.5ex] (\thesarrow) {$\scriptstyle #1$};
\path[draw,<-,decorate,
  decoration={zigzag,amplitude=0.7pt,segment length=1.2mm,pre=lineto,pre length=4pt}]
    (\thesarrow.south east) -- (\thesarrow.south west);
\end{tikzpicture}}%
}

\DeclareMathOperator{\env}{env}

\newcommand{\acts}{\righttoleftarrow}
\newcommand{\Af}{\mathfrak{A}}
\newcommand{\al}{\alpha}

\newcommand{\be}{\beta}

\newcommand{\De}{\Delta}
\newcommand\de{\delta}

\newcommand\eb{\boldsymbol{\epsilon}}
\newcommand{\Ec}{\mathcal E}
\newcommand\ep{\varepsilon}
\newcommand{\Ef}{\mathfrak{E}}

\newcommand{\Gb}{\mathbf{G}}
\newcommand{\gb}{\mathbf{g}}
\newcommand\Gd{\Gb^{\,(2)}}

\newcommand{\ka}{\varkappa}

\newcommand{\la}{\lambda}
\newcommand{\lab}{\boldsymbol{\lambda}}
\newcommand\lb[1]{\mathord{\upharpoonright}#1\mathord{\upharpoonleft}}
\newcommand\lbb[1]{\mathord{\raisebox{-.5pt}{$\mathlarger{\mathlarger{\upharpoonright}}$}}#1\mathord{\raisebox{-.5pt}{$\mathlarger{\mathlarger{\upharpoonleft}}$}}}

\newcommand{\mapstoto}{\mathop{\,\mathrel{\mapstochar}\kern2pt\joinrel\relbar\kern-1pt\joinrel\relbar\joinrel\rightsquigarrow\,}}

\newcommand{\mf}{\mathfrak{m}}
\newcommand\Mor{\operatorname{\mathsf{Mor}}}
\newcommand{\mub}{\boldsymbol{\mu}}
\newcommand\myp{\mkern-.0mu\raise0.4ex\hbox{$\scriptstyle\prime$}}

\newcommand\no[1]{\textbf{#1$^{\mathbf{o}}$.}}
\newcommand\nm[1]{\textbf{#1$^{\mathbf{o}}$}}

\newcommand{\Obj}{\operatorname{\mathsf{Obj}}}
\newcommand{\Om}{\Omega}
\newcommand{\om}{\omega}
\newcommand{\ov}{\overline}

\newcommand{\Pb}{\mathsf P}
\newcommand{\Pc}{\mathcal P}
\newcommand{\Pca}{\Pc_{\la}}
\newcommand{\pib}{\boldsymbol{\pi}}
\newcommand\pig[1]{\scalerel*[5pt]{\big#1}{%
  \ensurestackMath{\addstackgap[1.5pt]{\big#1}}}}
\newcommand\pigl[1]{\mathopen{\pig{#1}}}
\newcommand\pigr[1]{\mathclose{\pig{#1}}}

\newcommand{\si}{\sigma}
\newcommand\so{\operatorname{\mathsf{s}}}

\newcommand\ta{\operatorname{\mathsf{t}}}
\renewcommand{\th}{\theta}
\newcommand{\thb}{{\boldsymbol{\theta}}}

\newcommand{\un}{\underline}
\newcommand{\uun}[1]{\raisemath{1.1pt}{\un{#1}}}

\newcommand{\ZZ}{\mathbb Z}

\begin{document}

\title[Amenability of groupoids]{Amenability of groupoids and \\ asymptotic invariance \\ of convolution powers}

\author{Theo B\"uhler}
\address{T.B.: Z\"urich, Switzerland}
\email{math@theobuehler.org}

\author{Vadim A.\ Kaimanovich}
\address{V.K.: Department of Mathematics and Statistics, University of
Ottawa, 150 Louis Pasteur, Ottawa ON, K1N 6N5, Canada}
\email{vkaimano@uottawa.ca, vadim.kaimanovich@gmail.com}

\begin{abstract}
The original definition of amenability given by von Neumann in the highly non-constructive terms of means was later recast by Day using approximately invariant probability measures. Moreover, as it was conjectured by Furstenberg and proved by Kaimanovich\,--\,Vershik and Rosenblatt, the amenability of a locally compact group is actually equivalent to the existence of a single probability measure on the group with the property that the sequence of its convolution powers is asymptotically invariant. In the present article we extend this characterization of amenability to measured groupoids. It implies, in particular, that the amenability of a measure class preserving group action is equivalent to the existence of a random environment on the group parameterized by the action space, and such that the tail of the random walk in almost every environment is trivial.
\end{abstract}

\maketitle

\thispagestyle{empty}

\epigraph{And though absolute justice \newline be unattainable, as much justice \newline as we need for all practical use \newline is attainable by all those \newline who make it their aim.}%
{\textit{Unto This Last}\\ \textsc{John Ruskin}}

\section*{Introduction}

\no1 The origins of amenability\,---\,Lebesgue's and Hausdorff's measure problems, the paradoxes of Hausdorff and Banach\,--\,Tarski, the solution by Banach of the Hausdorff problem in low dimensions\,---\,left their distinctive hallmark during the early years. Von Neumann (who formulated his definition in 1929 \cite{vonNeumann29}) and other pioneers were, of course, perfectly aware of the propinquity between the Banach\,--\,Mazur limits and Ces\`aro's averaging, yet they felt no need to leave the heights of newfound Cantor's paradise.

Perfect invariance being unattainable in the lowly world of averages and probabilities, one can try to make do with an \emph{approximate} one. It is the latter\,---\,much more constructive\,---\,avenue that was explored from the late 40s through the mid-60s in several directions and resulted in a plethora of necessary and sufficient conditions for amenability.

The beginning of this new period is marked by the appearance of the English term \emph{amenability} coined by Day in 1949 \cite{Day49} (the full version
\cite{Day50} appears in 1950). In this announcement he introduces the following ``strong amenability'' condition on a \emph{discrete group} $G$: there exists a net (a sequence, if $G$ is countable) of finitely supported probability measures~$\th_n$ on $G$ which is (strongly) \emph{asymptotically invariant} in the sense that
\begin{equation} \label{eq:day}
\| g\th_n - \th_n \| \to 0 \qquad \forall\, g\in G \;.
\end{equation}
The fact that this ``strong'' amenability is actually equivalent to the ``plain'' one is established by Day in his 1957 follow-up paper \cite[Theorem~1]{Day57}.

In the meantime Reiter and F{\o}lner (independently both of Day and of each other) introduce their respective conditions as well \footnotemark.

\footnotetext{\,Although in modern expositions Day is somewhat eclipsed by Reiter and F{\o}lner, his contribution was undeniably recognized by the contemporaries. One reason for this shift might be that in spite of the prominent presence of the results from \cite{Day50, Day57} in Greenleaf's trendsetting book \cite{Greenleaf69}\,---\,namely, in \S 2.4 entitled \emph{The celebrated method of Day}, which is precisely how Day's work is referred to by Hulanicki \cite[p.\ 88]{Hulanicki66}\,---\,formula \eqref{eq:day} explicitly appears in \cite{Greenleaf69} only in \S 3.2 called \emph{Reiter's work in harmonic analysis (Reiter's condition)}.}

F{\o}lner's condition \cite[Main theorem (1)]{Folner55} is essentially an ``$\ep$-form'' of Day's condition~\eqref{eq:day} with the important difference, though, that the measures $\th_n$ are \linebreak additionally required to be uniform on their supports (as this condition, unlike~\eqref{eq:day}, is formulated in terms of finite subsets of the group). Its predecessor (formulated using the Haar measure of the involved sets) appeared in Dixmier's earlier paper \cite[p. 221]{Dixmier50} as a sufficient condition for the existence of an invariant mean on the $L^\infty$ space of a locally compact group. F{\o}lner\,---\,unaware of both \cite{Day50} and \cite{Dixmier50}\,---\,works in the same setup of discrete groups as von Neumann and Day, and proves the equivalence of his condition to amenability. Day, in addition to his own argument, mentions in \cite{Day57} that the characterization of amenability by condition \eqref{eq:day} is also a consequence of F{\o}lner's criterion.

Reiter's condition for locally compact \emph{topological groups}\,---\,the uniform convergence on compact sets in \eqref{eq:day} with absolutely continuous measures $\th_n$\,---\,first appears in a rather cumbersome notation in his paper \cite[formula (ii$'$) on~p.~405]{Reiter52}, \linebreak where it is established for abelian groups. In 1960 it is more explicitly restated in \cite[Lemma~1]{Reiter60} and popularized by Dieudonn\'e \cite[p.~284]{Dieudonne60} under the telling name ``property~(P$_1$)''. At the time neither of them is aware of any links with amenability; it is first mentioned by Reiter in the 1965 paper \cite{Reiter65} where he proves the equivalence of condition (P$_1$) to the existence of a topological invariant mean on the group.

\medskip

\no2 In the late 70s the notion of amenability is extended \emph{beyond groups} to actions, equivalence relations, and, more generally, to \emph{groupoids} (see Connes \cite{Connes94} and Weinstein \cite{Weinstein96} for their physical, algebraic, geometrical and analytical \emph{raisons d'\^etre} and a passionate apology for the groupoid cause).

The forerunners of this development are a number of examples in which actions of non-amenable groups\,---\,like, for instance, the boundary action of the group $PSL(2,\ZZ)$, see Bowen \cite{Bowen77} and Vershik \cite{Vershik78}\,---\,exhibit properties similar to amenability. The original definition of an \emph{amenable action} is given by Zimmer \cite[Definition 1.4]{Zimmer78} using a rather involved fixed point property (which is the main \emph{application} of amenability, see footnote 2 on next page, although this property is quite inconvenient for \emph{establishing} amenability). It is further carried over to \emph{equivalence relations} and reformulated in terms of means in \cite[\mbox{Proposition~4.1}]{Zimmer77}. The fact that \emph{groupoids} provide a natural setup for Zimmer's definition is immediately pointed out already in the review of his paper \cite{Combes77}. A systematic implementation of this is done by Renault \cite{Renault80} who also establishes several approximate invariance conditions equivalent to the amenability of a \emph{measured groupoid}.

\medskip

\no3 Furstenberg, who introduces the \emph{Poisson boundary} of a random walk in 1963, immediately notices and uses its relationship with amenability \cite{Furstenberg63a, Furstenberg63} \footnotemark. Yet the fact that the \emph{Liouville property} of a random walk ($\equiv$ the \mbox{absence} of non-constant bounded harmonic functions $\equiv$ the triviality of the \mbox{Poisson} \mbox{boundary}) implies the amenability of the underlying group is first explicitly stated several years later by Azencott \cite[Proposition 1]{Azencott69}, \cite[Proposition II.1 and the Corollaire on p.~43]{Azencott70}\,---\,if heavily relying on the work of Furstenberg (a more direct argument in \cite[Section 9]{Furstenberg73} is also based on the fixed point property characterization of amenability). On the other hand, amenable groups may also carry non-Liouville random walks, and Furstenberg conjectures \cite[p.~213]{Furstenberg73} that a group is amenable if and only if it admits a Liouville random walk.

\footnotetext{\,Apparently, Furstenberg was not aware of the notion of amenability at the time of writing \cite{Furstenberg63a, Furstenberg63}. Instead, he argued in terms of the \emph{fixed point property} introduced by him \emph{\`a la} Markov\,--\,Kakutani in  \cite[Definition~1.4]{Furstenberg63}\,---\,the existence of a fixed point for any affine action of a given group on a compact convex set in a locally convex linear topological space. Actually, this property had already been considered in 1961 by Day \cite{Day61} who showed that it follows from the amenability of the group. Moreover, this implication is essentially contained in the 1939 note of Bogolyubov \cite{Bogolyubov39} that remained virtually unknown until the 90s, see Anosov \cite{Anosov94} and Grigorchuk\,--\,de~la~Harpe \cite{Grigorchuk-delaHarpe17}. The proof of the converse implication\,---\,that the Markov\,--\,Kakutani fixed point property implies amenability\,---\,indicated by Day in \cite[Theorem 3]{Day61} in the generality of topological groups contains a gap (and in the original form only works for discrete groups) ultimately filled in by Rickert in 1967 \cite{Rickert67}.}

In view of Day's condition \eqref{eq:day}, a much more direct link with amenability is provided by the general fact, a consequence of Derriennic's \emph{0-2 laws} \cite{Derriennic76}, that for any Markov chain the Liouville property is equivalent to the asymptotic independence of its one-dimensional distributions (or their Ces\`aro averages, to be more precise) of the initial position of the chain. Therefore, in the group case the Liouville property not only implies amenability, but also gives rise to an explicit asymptotically invariant sequence \eqref{eq:day} in the most ``economical'' way: one just has to provide a single measure~$\mu$ (the step distribution of the random walk), and then the whole sequence is obtained from $\mu$ by using the convolution operation provided by the intrinsic group structure, see Kaimanovich\,--\,Vershik \cite[Theorem~4.2]{Kaimanovich-Vershik83}.

An equally direct (if entirely non-constructive) approach consists in using a \emph{measure-linear mean} to provide an equivariant projection from the space of bounded functions on the state space onto constants, see Connes\,--\,Feldman\,--\,Weiss \cite[Proposition 20]{Connes-Feldman-Weiss81}, Lyons\,--\,Sullivan \cite[Theorem 3$'$]{Lyons-Sullivan84}, Kaimanovich\,--\,Fisher \cite[Theorem~1]{Kaimanovich-Fisher98}.

Furstenberg's conjecture was proved by Vershik and the second author \cite{Vershik-Kaimanovich79,Kaimanovich-Vershik83}, and, independently, by Rosenblatt \cite{Rosenblatt81} by constructing, for any locally compact amenable group $G$, an absolutely continuous probability measure~$\mu$ on $G$ with the property that the sequence of its convolution powers satisfies Day's condition~\eqref{eq:day}. The proof in \cite{Kaimanovich-Vershik83} uses Reiter's condition (see Forghani\,--\,Kaimanovich \cite{Forghani-Kaimanovich15p} for its interpretation in the probabilistic language of stopping times), whereas a similar but more complicated argument in \cite{Rosenblatt81} is based on a topological analogue of F{\o}lner's criterion.

\medskip

\no4 The common features of a number of examples which still demonstrate a certain stochastic homogeneity in spite of not being space homogeneous \emph{sensu stricto}\,---\,the Brownian motion on foliations, the random walks in random environment and on equivalence relations\,---\,prompted the second author to introduce the general setup of \emph{equivariant Markov chains on groupoids} \cite{Kaimanovich05a}. In the same way as in the particular case of groups (see \nm3 above), the Liouville property for such chains almost automatically implies the amenability of the underlying groupoid. A natural question then was to ask whether Furstenberg's conjecture holds for groupoids as well, namely, whether \emph{any amenable groupoid carries a Liouville equivariant Markov chain} \cite[Conjecture 4.6]{Kaimanovich05a}.

The purpose of this paper is to present a \emph{proof of the above conjecture for measured groupoids}. A preliminary draft was written by the first author in 2006 \cite{Buhler06p}, and this result was presented at a number of seminars at the time. The current version has been rewritten by the second author. In the meantime this conjecture was independently proved by Chu\,--\,Li \cite{Chu-Li18} in a greater generality of semigroupoids, both in the measure and in the topological categories. Although our approach and that of \cite{Chu-Li18} are based on the same general idea from \cite{Kaimanovich-Vershik83}, technically they are quite different, and our argument appears to be more straightforward.

The construction of a Liouville measure from \cite[Theorem 4.3]{Kaimanovich-Vershik83} was carried over to numerous other setups by Hayashi\,--\,Yamagami \cite[Theorem~2.5]{Hayashi-Yamagami00}, the second author \cite[Theorem 1]{Kaimanovich02a}, Bartholdi \cite[Theorem~8.20]{Bartholdi18}, Juschenko\,--\,Zheng \cite[Lemma~2]{Juschenko-Zheng18}, Schneider\,--\,Thom \cite[Theorem~4.8]{Schneider-Thom19p}. We expect our technique to provide a uniform treatment of all these situations and be applicable to other Markov chains (in particular, on hypergroups and hypergroupoids) as well.

\medskip

\no5 Let us briefly outline the structure of the paper. In \secref{sec:1} we remind the reader of the definition of amenable measured groupoids and of the associated background.

One can think about a groupoid $\Gb$ (the most succinctly defined as a small category with invertible morphisms) as a collection of arrows (morphisms) between its objects. There is a composition operation that obeys the same rules as the group multipication (groups being precisely the groupoids with only one object). An important difference, though, is that the composition in groupoids is partial in the sense that two arrows are composable only if the target of one matches the source of the other one. We denote by $\Gb^x$ the fibre of the target map over an object $x$, i.e., the set of all arrows with the target~$x$. These fibres are moved around according to the formula
$$
\gb \,\Gb^{\,\uun{\gb}} = \Gb^{\,\ov{\gb}} \qquad\forall\,\gb\in\Gb \;,
$$
where $\un{\gb}$ and $\ov{\gb}$ are the source and the target of the morphism $\gb$, respectively. A~Haar system is a collection $\lab = \{\la^x\}_{x\in X}$ of measures on the fibres $\Gb^x$ which is left invariant in the sense that
$$
\gb \, \la^{\uun{\gb}} = \la^{\ov{\gb}} \qquad\forall\,\gb\in\Gb \;,
$$
The remaining ingredient of the definition of a measured groupoid $(\Gb,\lab,\ka)$ is a measure~$\ka$ on the set of objects which is required to be quasi-invariant in a sense similar to the quasi-invariance used for group actions. By integrating the measures~$\la^x$ from the Haar system against $\ka$ one obtains then a measure on the whole groupoid denoted by $\lab \star \ka$.

The definition of the amenability of a measured groupoid $(\Gb,\lab,\ka)$ we are working with (\dfnref{dfn:reiter}) is the one of Renault \cite[Lemma 3.4]{Renault80} recast as in Anantharaman-Delaroche\,--\,Renault \cite[Definition 2.6]{Anantharaman-Renault01} and \cite[\mbox{Section~2.B}]{Kaimanovich05a}. It requires what we call the \emph{integral strong asymptotic invariance (ISAI)} condition: there exists a sequence of systems $\left\{\th_n^x\right\}_{x\in X}$ of absolutely continuous probability measures on the fibres $\Gb^x$ such that the \emph{discrepancy functions}
\begin{equation} \label{eq:di}
\gb \mapsto \left\| \gb\,\theta_{n}^{\uun{\gb}} - \theta_{n}^{\ov{\gb}} \right\|
\end{equation}
converge to 0 in the weak$^*$ topology $\si(L^\infty,L^1)$ of the space $L^\infty(\Gb,\lab \star \ka)$.

The first rather obvious observation (\remref{rem:mw}) consists in noticing that in our situation the weak$^*$ convergence is equivalent just to the convergence
$$
\int \left\| \gb\,\theta_{n}^{\uun{\gb}} - \theta_{n}^{\ov{\gb}} \right\| \, d\mf (\gb) \xrightarrow[n \to \infty]{} 0
$$
of the integrals with respect to a \emph{single} reference probability measure $\mf$ equivalent to $\lab\star\ka$, or, in the ``$\ep$-form'',
$$
\forall\,\ep>0 \quad \exists\,\{\th^x\}: \quad \int \left\| \gb\,\theta^{\uun{\gb}} - \theta^{\ov{\gb}} \right\| \, d\mf (\gb) < \ep \;.
$$

\medskip

\no6 The principal part of the paper is \secref{sec:2}. We begin it with recalling the definition of an equivariant Markov chain on a groupoid $\Gb$ from \cite{Kaimanovich05a} (\dfnref{dfn:chain}). One imposes on its transition probabilities the same equivaraince condition
$$
\pi^{\gb'\gb} = \gb'\,\pi^{\gb}
$$
as for random walks on groups, but in the groupoid case additionally one has to make sure that the translate $\gb'\,\pi^{\gb}$ is well-defined whenever the morphisms $\gb$ and $\gb'$ are composable, which means that the transition probability $\pi^\gb$ has to be concentrated on the fibre $\Gb^{\,\ov{\gb}}$. Thus, the sample paths of an equivariant Markov chain are confined to the fibres of the target map, so that the whole ``global chain'' can be considered as a collection of ``local'' chains on the fibres.

The set of objects of a groupoid is embedded into the set of morphisms by the unit inclusion map. Therefore (\prpref{prp:mf}), an equivariant Markov chain gives rise to a target fibred system of probability measures, and, conversely, all transition probabilities can be recovered from this system by the formula
$$
\pi^{\gb}=\gb\,\pi^{\uun{\gb}} \;.
$$
The product of two equivariant chains (more rigorously, of their transition operators) corresponds to the usual convolution of the associated target fibred systems of measures (or, in the absolutely continuous case, of their densities).

The discrepancy functions \eqref{eq:di} from the definition of an amenable groupoid can be then interpreted in terms of equivariant Markov chains as
$$
\left\| \gb\,\theta^{\uun{\gb}} - \theta^{\ov{\gb}} \right\|
= \left\| \pi^\gb - \pi^{\ov \gb} \right\|
= \left\| (\de_\gb - \de_{\ov \gb}) P \right\| = \De(\gb,P) \;,
$$
where $\{\th^x\}$ is a system of measures on the target map fibres,
$$
\pi^\gb=g\th^{\uun{\gb}}
$$
are the transition probabilities of the associated equivariant Markov chain, and $P$ is the transition operator of this chain. We say that
$$
\gb\mapsto\De(\gb,P)
$$
is the \emph{discrepancy function of an equivariant transition operator} $P$, or, of the corresponding equivariant Markov chain (\dfnref{dfn:dscrp}). As we have already mentioned, the sample paths of an equivariant Markov chain are confined to the fibres of the target map, and the value of the discrepancy function at a point $\gb$ is the total variation distance between the transitions probabilities issued from $\gb$ and from the distinguished point $\ov \gb$ of its target fibre~$\Gb^{\,\ov \gb}$. The \emph{mean discrepancy} with respect to a probability measure $m$ on $\Gb$ is the $m$-average
$$
\De(m,P) = \int \De(\gb,P)\,dm(\gb)
$$
of the discrepancy function.

We can now say that a sequence of equivariant transition operators $P_n$ on groupoid satisfies the \emph{integral strong asymptotic invariance (ISAI)} condition\,---\,it was earlier introduced for target fibred systems of measures\,---\,if their discrepancy functions converge to~0 in the weak$^*$ topology, or, equivalently, if
$$
\De(\mf,P_n) \to 0
$$
for a fixed reference measure $\mf$ (\dfnref{dfn:IDRop}). Renault's definition of an amenable measured groupoid amounts then to the existence of an ISAI sequence of equivariant Markov operators.

\medskip

\no7 In the group case an equivariant Markov operator $P=P_\th$ is determined just by a single probability measure $\th$ on the group, and its mean discrepancy is then
\begin{equation} \label{eq:md}
\De(m,P) = \De(m,\th) = \int \| g\th - \th \| \,dm(g) \;,
\end{equation}
so that the amenability of a locally compact group is equivalent to the existence of a sequence of absolutely continuous probability measures $\th_n$ such that
\begin{equation} \label{eq:1}
\De(\mf,\th_n) \to 0
\end{equation}
for a fixed ($\equiv$ any) reference probability measure $\mf$ equivalent to the Haar measure, or, in the $\ep$-form,
\begin{equation} \label{eq:2}
\forall\,\ep>0 \quad \exists\, \th: \quad \De(\mf,\th) < \ep \;.
\end{equation}

Although in the group context one usually employs the criteria of Day, F{\o}lner, or Reiter (see \nm1 above), conditions \eqref{eq:1} and \eqref{eq:2} can be quite useful as well. Condition \eqref{eq:2} is nothing but an extension of the usual F{\o}lner condition from sets to measures, where the reference measure~$\mf$ plays the role of a generating set (see \cite{Hulanicki66} and \cite{Kaimanovich92b} for passing from sets to measures and back in isoperimetric conditions). A similar integral characteristic obtained by replacing the total variation with the Kullback\,--\,Leibler divergence plays an important role in the entropy theory of random walks developed in \cite{Kaimanovich-Vershik83}.

Actually, the mean discrepancy \eqref{eq:md} in a hidden form appears in the second of the two F{\o}lner criteria (the ``sufficient condition'' of \cite[Main Theorem (2)]{Folner55}), which in our notation states that the amenability of a discrete group $G$ is equivalent to the existence of a constant $C<2$ such that for any finitely supported probability measure $\mf$ on $G$ there is a finite subset $A\subset G$ with
\begin{equation} \label{eq:fo2}
\De(\mf,\upchi_A) < C \;,
\end{equation}
where $\upchi_A$ denotes the uniform probability measure on $A$, cf. Juschenko\,--\,Nagnibeda \cite[Theorem 17]{Juschenko-Nagnibeda15} and Gournay \cite[Corollary 5.2]{Gournay15}. A topological version of condition~\eqref{eq:fo2} was recently obtained by Schneider\,--\,Thom \cite[Theorem~4.5(3)]{Schneider-Thom18}. It would be interesting to extend \eqref{eq:fo2} to measures and to carry it over to topological groups.

\medskip

\no8 The key ingredient of our approach is the following simple observation:

\begin{namedthm*}{\prpref{prp:QPn}}
If $(P_n)$ is an ISAI sequence of transition operators on a measured groupoid, then for any other equivariant transition operator with absolutely continuous transition probabilities $Q$ the sequence $(QP_n)$ is also ISAI.
\end{namedthm*}

By using it we establish our main technical tool:

\begin{namedthm*}{\thmref{thm:IDRpower}}
If $(P_n)$ is an ISAI sequence of equivariant transition operators with absolutely continuous transition probabilities on a measured groupoid, then there exists an infinite convex combination $P$ of the operators $P_n$ such that the sequence of powers $P^n$ is also ISAI.
\end{namedthm*}

Its proof is based on the same general idea as in \cite[Theorem~4.3]{Kaimanovich-Vershik83}. However, the use of \prpref{prp:QPn} instead of Reiter's condition allows us both to simplify and to generalize that argument. As an immediate corollary of \thmref{thm:IDRpower} we then obtain our principal result:

\begin{namedthm*}{\thmref{thm:powers}}
A measured groupoid $(\Gb,\lab,\ka)$ is amenable if and only if there exists an equivariant transition operator on $P$ with absolutely continuous transition probabilities such that the sequence of its powers satisfies the integral strong asymptotic invariance condition (ISAI).
\end{namedthm*}

It implies, in particular, that the weak$^*$ convergence to 0 of the discrepancy functions \eqref{eq:di} in the definition of amenability can be replaced with the convergence almost everywhere.

By using the general 0-2 laws for Markov chains \thmref{thm:IDRpower} is further interpreted in terms of the triviality of the tail and of the Poisson boundaries of the fibrewise Markov chains determined by equivariant transition operators (we recall that since the samples paths of an equivariant Markov chain are confined to the target map fibres, it only makes sense to talk about the boundary triviality of the fibrewise chains). Namely,

\begin{namedthm*}{\thmref{thm:main}}
The amenability of a measured groupoid is equivalent to the existence of an equivariant Markov chain with absolutely continuous transition probabilities whose tail boundary (alternatively, Poisson boundary) is fibrewise trivial.
\end{namedthm*}

\medskip

\no9 In \secref{sec:3} we specialize the general results of \secref{sec:2} to the particular case of measure class preserving group actions. The fact that the action of any (not necessarily amenable) locally compact group on its Poisson boundary is in a sense similar to actions of amenable groups was the motivation for Zimmer's definition of an amenable action \cite{Zimmer78}, and the Poisson boundary action was his primary example.

Zimmer's argument for proving the amenability of the Poisson boundary action critically depends on the time homogeneity. By using a different approach (cf. the discussion in \nm3 above) Connes\,--\,Woods \cite[Theorem 2.2 and Remark 2.3]{Connes-Woods89} established the amenability of the action of a locally compact group $G$ on the Poisson boundary for what they called \emph{matrix-valued random walks} (equivariant Markov chains on the product of~$G$ by a certain countable set subject to some additional assumptions on the transition probabilities). Actually, since the matrix-valued random walks are time inhomogeneous, in this situation one should rather talk about the tail boundary (the quotient of the path space by the synchronous asymptotic equivalence relation), see \secref{sec:02}. It was later proved that, conversely, any measure class preserving amenable action of a second countable locally compact group can be presented as the boundary action of an appropriately defined matrix-valued random walk: first by Elliott\,--\,Giordano \cite{Elliott-Giordano93} for countable groups and then by Adams\,--\,Elliott\,--\,Giordano \cite{Adams-Elliott-Giordano94} in the general case.

Given a group action $G:X\acts$, the objects of the associated groupoid $\Gb$ are just the points of the action space $X$. The morphisms of $\Gb$ are the pairs of points from the same action orbit labelled with a group element that moves one to the other. Therefore, the fibres $\Gb^x$ can be identified with the group $G$, and the systems of measures $\{\th^x\}_{x\in X}$ on the fibres $\Gb^x$ that appear in the definition of amenability are nothing but maps from the action space $X$ to the space of probability measures on $G$.

Renault's definition mentioned in \nm5 takes then the following form for a measure class preserving action of a locally compact group $G$ on a measure space $(X,\ka)$: the action is amenable if and only if there exists a sequence of measurable maps $x\mapsto\th_n^x$ from $X$ to the space $\Pca(G)$ of absolutely continuous probability measures on $G$ which is asymptotically equivariant in the sense that the functions
$$
(g,x) \mapsto \|g\th_n^x - \th_n^{gx} \|
$$
converge to 0 in the weak$^*$ topology of the space $L^\infty(G\times X, \la\otimes\ka)$; here $\la$ is a Haar measure on $G$.

Equivariant Markov chains on the action groupoid $\Gb$ can be interpreted as\,---\,in the probabilistic terminology\,---\,\emph{random walks in random environment} on the group~$G$. The presence of a group structure allows one to talk about the increments between consecutive group elements, and to describe Markov chains on $G$ in terms of the distibutions of these increments. Thus, an \emph{environment} on $G$ is a collection $\mub=\{\mu^g\}_{g\in G}$ of probability measures on $G$ which determines the Markov chain with the transitions
$$
g\xrsquigarrow{h\sim\mu^g} gh
$$
(the ordinary space homogeneous random walk determined by a step distribution~$\mu$ corresponds to the constant environment $\mu^g\equiv\mu$). A random walk in random environment consists in sampling first an environment $\mub$, and then performing the random walk in this environment.

The restriction of the equivariant Markov chain determined by a system of measures~$\{\th^x\}$ to a single fibre $\Gb^x$ can, therefore, be considered as the random walk on $G$ in the environment
$$
\mub=\{\mu^g\}_{g\in G} \;, \qquad \mu^g = \th^{g^{-1} x} \;,
$$
which becomes ``random'' in the presence of a probability measure on the action space. Thus, for actions
\thmref{thm:main} can be restated in the following way:

\pagebreak

\begin{namedthm*}{\thmref{thm:actions}}
The amenability of a measure class preserving action of a locally compact group $G$ is equivalent to the existence of a measurable map $x\mapsto\th^x$ from the action space to the space of absolutely continuous probability measures on $G$ such that the tail boundary (alternatively, the Poisson boundary) of the random walk on the group $G$ in the arising random environment is almost surely trivial.
\end{namedthm*}

\section{Amenable groupoids} \label{sec:1}

\subsection{Groupoids} \label{sec:gr}

A \textsf{groupoid} is a \emph{small category in which each morphism is an isomorphism}. In other words, a groupoid $\Gb$ is determined by a \textsf{set of objects} (\textsf{units})
$$
X=\Obj \Gb=\Gb^{\,(0)}
$$
and a \textsf{set of morphisms} (\textsf{elements}) denoted just by
$$
\Gb \cong \Mor \Gb = \Gb^{\,(1)}
$$
endowed with the \textsf{source} (\textsf{domain}) and \textsf{target} (\textsf{codomain}, or \textsf{range}) maps
$$
\so(\gb) = \un{\gb} \;, \qquad \ta(\gb) = \ov{\gb}
$$
from $\Gb$ to $X$. The set of \textsf{composable pairs} in $\Gb$ is
$$
\Gd = \left\{ (\gb,\gb\myp)\in \Gb\times\Gb: \un{\gb}=\ov{\gb\myp} \right\} \;,
$$
and the \textsf{composition} (\textsf{multiplication}) is a map
$$
\Gd\to \Gb \;, \qquad (\gb,\gb\myp)\mapsto \gb\,\gb\myp
$$
such that
\begin{equation} \label{eq:mult}
\un{\gb\, \gb\myp} = \un{\gb\myp} \;,\qquad \ov{\gb\, \gb\myp}=\ov{\gb} \;,
\end{equation}
see \figref{fig:groupoid}.

\begin{figure}[h]
\begin{center}
\psfrag{a}[][]{$\un{\gb\myp}$}
\psfrag{c}[][]{$\un{\gb}=\ov{\gb\myp}$}
\psfrag{e}[][]{$\ov{\gb}$}
\psfrag{u}[cl][cl]{$\gb\myp$}
\psfrag{v}[cl][cl]{$\gb$}
\psfrag{x}[ct][ct]{$\;\;\gb\, \gb\myp$}
\includegraphics[scale=.8]{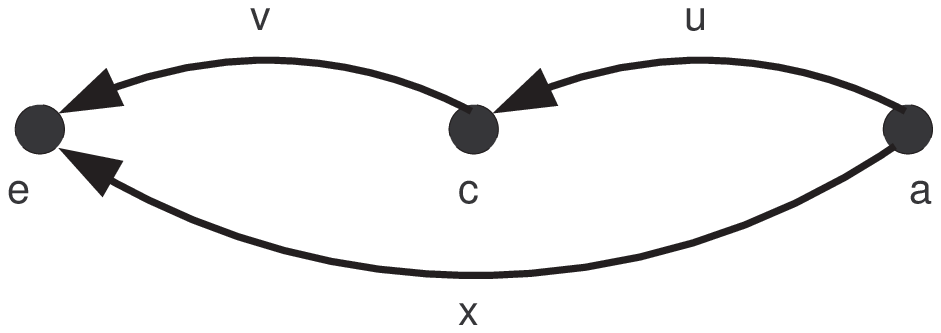}
\end{center}
\caption{Composition in groupoids.}
\label{fig:groupoid}
\end{figure}

\begin{rem}
The multiplication rule \eqref{eq:mult} matches the prefix notation used for the left actions of groups, in which one has $(gg')x=g(g'x)$, i.e., $g'$ is applied ``first'', cf. \secref{sec:3}. This is why the arrows in \figref{fig:groupoid} are directed to the left. This convention allows us to introduce the equivariant Markov chains on a groupoid (\dfnref{dfn:chain}) by imposing the same condition $\pi^{\gb\gb'} = \gb\pi^{\gb'}$ as in the usual definition of the (right) random walks on groups.
\end{rem}

The \textsf{fibres} of the source and the target maps are denoted
$$
\Gb_{\,x} = \so^{-1}(x) = \left\{\gb\in \Gb: \un{\gb}=x \right\} \;, \qquad
\Gb^{\,x} = \ta^{-1}(x) = \left\{\gb\in \Gb: \ov{\gb}=x \right\} \;,
$$
respectively, so that for any $\gb\in\Gb$ the domain of the left multiplication by $\gb$ is precisely the fibre $\Gb^{\,\uun{\gb}}$, and
\begin{equation} \label{eq:gG}
\gb \,\Gb^{\,\uun{\gb}} = \Gb^{\,\ov{\gb}} \qquad\forall\,\gb\in\Gb \;.
\end{equation}

There is an embedding (the \textsf{unit inclusion})
$$
\eb:X\to \Gb \;,
$$
which associates to any object $x\in X$ the \textsf{identity morphism} $\eb_x$ (below we shall routinely identify $x$ and $\eb_x$) such that
$$
\un{\eb_x} = \ov{\eb_x} = x \qquad \forall\,x\in X
$$
and
$$
\gb \, \eb_{\un{\gb}} = \eb_{\ov{\gb}} \, \gb = \gb \qquad\forall\,\gb\in\Gb \;.
$$
For any $\gb\in\Gb$ there is a unique \textsf{inverse morphism} $\gb^{-1}$ with the property that
$$
\un{\gb^{-1}} = \ov{\gb} \;, \qquad \ov{\gb^{-1}} = \un{\gb} \;,
$$
and
$$
\gb\,\gb^{-1} = \eb_{\ov{\gb}} \;, \qquad \gb^{-1}\,\gb = \eb_{\un{\gb}} \;.
$$
Finally, the composition (when well-defined) is associative.

\medskip

The standard \emph{examples of groupoids} are
\begin{enumerate}[---]
\item
\emph{groups}: they are the groupoids with only one object;
\item
\emph{equivalence relations} (the associated groupoids are called \textsf{principal}): the morphisms are just the pairs of equivalent objects;
\item
\emph{group actions} (which combine the features both of groups and of equivalence relations): the morphisms are the triples which consist of two action space points together with a group element moving one point to the other one, see \secref{sec:3} for more details.
\end{enumerate}

\subsection{Measured groupoids} \label{sec:am}

A \textsf{measurable structure} (i.e., a $\si$-algebra of \textsf{measurable sets}) on the set of morphisms of a groupoid $\Gb$ induces the pullback measurable structures both on the set of objects $X$ (by the unit inclusion map $\eb:X\hookrightarrow \Gb$) and on the set of composable pairs~$\Gd$ (by the inclusion $\Gd\hookrightarrow \Gb\times\Gb$). If all structure maps described in \secref{sec:gr} are measurable, then the groupoid~$\Gb$ is called \textsf{measurable}.

A \textsf{measurable system of measures}
$$
\lab = \{\la^x\}_{x\in X}
$$
on the fibres $\Gb^{\,x}$ of the target map of a measurable groupoid $\Gb$ (in short: a \textsf{measurable target fibred system of measures}) is the one for which the map
$$
x \mapsto \langle f, \la^x \rangle
$$
is measurable whenever $f$ is a non-negative measurable function on $\Gb$ (the fibrewise measures $\la^x$ are not required to be normalized or finite unless otherwise specified). A \textsf{Haar system} on a measurable groupoid $\Gb$ is a measurable target fibred system of measures $\lab = \{\la^x\}_{x\in X}$ which is \textsf{left invariant} in the sense that
$$
\gb \, \la^{\uun{\gb}} = \la^{\ov{\gb}} \qquad\forall\, \gb\in\Gb \;,
$$
and \textsf{proper}, i.e., there exists a non-negative measurable function $f$ on $\Gb$ such that
$$
\langle f, \la^x \rangle = 1 \qquad \forall\, x \in X \;.
$$

Given a measurable target fibred system $\lab=\{\la^x\}$ on a groupoid $\Gb$ and a measure~$\ka$ on its set of objects $X$, one can integrate the fibrewise measures $\la^x$ with respect to $\ka$ to obtain a measure on $\Gb$ denoted $\lab \star \ka$. The measure $\ka$ is called \textsf{quasi-invariant} with respect to a Haar system $\lab$ if the inversion map $g\mapsto g^{-1}$ preserves the $(\lab \star \ka)$-negligible sets. Finally, a \textsf{measured groupoid} is a triple $(\Gb,\lab,\ka)$, where $\Gb$ is a measurable groupoid, $\lab$~is a Haar system on $\Gb$, and $\ka$ is a quasi-invariant measure on the set of objects $X$.

\medskip

\conv{\em All measure spaces are assumed to be Lebesgue\,--\,Rokhlin measure spaces, and all measures are finite or $\si$-finite unless otherwise specified.}

\subsection{Amenability} \label{sec:amen}

Without spelling out numerous equivalent definitions of amenability of measured groupoids (for more details see Anantharaman-Delaroche\,--\,Renault \cite[Chapter 3]{Anantharaman-Renault00}), we shall only define it in terms of \textsf{asymptotically invariant target fibred systems of probability measures} (also called \textsf{approximate invariant means}), which is Renault \cite[Lemma 3.4]{Renault80} or Anantharaman-Delaroche\,--\,Renault \cite[Definition 2.6]{Anantharaman-Renault01} recast as in \cite[Section~2.B]{Kaimanovich05a}:

\begin{dfn}\label{dfn:reiter}
A measured groupoid $(\Gb,\lab,\ka)$ is called \textsf{amenable} if there exists a sequence of measurable target fibred systems $\thb_n=\left\{\th_n^x\right\}_{x\in X}$ of \textsf{absolutely continuous} probability measures (i.e., such that $\theta_{n}^x \prec \la^x$ for all $x \in X$) which satisfies the \textsf{integral strong \footnotemark\ asymptotic invariance (ISAI) condition}:
$$
\int \left\| \gb\,\theta_{n}^{\uun{\gb}} - \theta_{n}^{\ov{\gb}} \right\| \, dm (\gb) \xrightarrow[n \to \infty]{} 0
$$
for any probability measure $m \prec \lab \star \ka$,  where $\|\,\cdot\,\|$ denotes the total variation norm.
\end{dfn}

\footnotetext{\,The qualifier \emph{strong} in this definition refers to the \emph{strong topology} induced by the total variation distance on the space of measures.}

\begin{rem} \label{rem:mw}
The integral strong asymptotic invariance condition means that the sequence of \textsf{discrepancy functions}
\begin{equation} \label{eq:fIDR}
\gb \mapsto \left\| \gb \theta_{n}^{\uun{\gb}} - \theta_{n}^{\ov{\gb}} \right\| \;,\qquad \gb\in \Gb\;,
\end{equation}
converges to 0 in the \textsf{weak$^*$ topology} $\si(L^\infty,L^1)$ of the space $L^\infty(\Gb,\lab\star\ka)$. It is well-known and easy to verify (e.g., see Dunford\,--\,Schwartz \cite[Exercise~IV.13.27]{Dunford-Schwartz88}) that in our situation (convergence of non-negative functions with uniformly bounded $L^\infty$ norms to 0) the weak$^*$ convergence is equivalent just to the convergence of the integrals with respect to a \emph{single} probability measure from the same measure class. Thus, the integral strong asymptotic invariance condition on a sequence $(\thb_n)$ is equivalent to the convergence
$$
\int \left\| \gb\,\theta_{n}^{\uun{\gb}} - \theta_{n}^{\ov{\gb}} \right\| \, d\mf (\gb) \xrightarrow[n \to \infty]{} 0
$$
for a fixed reference probability measure $\mf$ equivalent to $\lab\star\ka$ (notation: $\mf\sim\lab\star\ka$).
\end{rem}

\begin{rem} \label{rem:pw}
Since the discrepancy functions \eqref{eq:fIDR} are uniformly bounded, the integral strong asymptotic invariance condition follows from the \textsf{pointwise strong asymptotic invariance condition}
$$
\left\| \gb\,\theta_{n}^{\uun{\gb}} - \theta_{n}^{\ov{\gb}} \right\| \xrightarrow[n \to \infty]{} 0 \qquad \text{for}\;(\lab\star\ka)\text{-a.e.} \;\gb\in\Gb \;,
$$
which is \emph{a priory} stronger. However, the integral and the pointwise conditions are equivalent for \emph{power sequences}, see the proof of \thmref{thm:powers} below.
\end{rem}

\section{Markov chains on groupoids and approximate invariance} \label{sec:2}

\subsection{Equivariant Markov chains}

We shall now interpret target fibred systems of probability measures on a groupoid in Markov terms.

\begin{dfn}[{\cite[Definition 3.1]{Kaimanovich05a}}] \label{dfn:chain}
A Markov chain on a measurable groupoid $\Gb$ determined by a measurable \footnotemark\, family of transition probability measures $\pib=\{\pi^\gb\}_{\gb\in\Gb}$ is called \textsf{equivariant} if
\begin{equation} \label{eq:tg}
\pi^\gb \pigl( {\Gb^{\,\ov{\gb}}} \pigr) = 1 \qquad\forall\,\gb\in \Gb \;,
\end{equation}
and
\begin{equation} \label{eq:g1g2}
\pi^{\gb\gb'} = \gb\,\pi^{\gb'} \qquad\forall\,(\gb,\gb\myp)\in\Gd \;,
\end{equation}
so that the map $\gb\mapsto\pi^\gb$ is equivariant with respect to the left action of $\Gb$.
\end{dfn}

\footnotetext{\,The measurability of the family $\pib$ is understood here in the same weak sense as in \secref{sec:am}, i.e., that the map $\gb\mapsto \langle f, \pi^\gb \rangle$ is measurable for any measurable function $f$ on $\Gb$.}

\begin{rem} \label{rem:fibre}
Condition \eqref{eq:tg} of the above definition is necessary and sufficient for the translate $\gb\,\pi^{\gb'}$ from condition \eqref{eq:g1g2} to be well-defined for any $(\gb,\gb\myp)\in\Gd$. The fact that all transition measures $\pi^\gb$ are concentrated on the corresponding fibres $\Gb^{\,\ov{\gb}}$ means that the value of the target map remains constant under the Markov transitions, i.e., one can consider the ``global'' chain on $\Gb$ as a collection of the ``local'' chains on the fibres $\Gb^{\,x}$ of the target map.
\end{rem}

Another equivalent way of describing a Markov chain consists in using the associated \textsf{transition operator} $P$ which acts as
\begin{equation} \label{eq:op}
Pf (\gb) = \langle f, \pi^\gb \rangle \;, \qquad \al P=\int \pi^\gb\,d\al(\gb)
\end{equation}
on appropriately defined spaces of functions and measures, respectively.

\begin{prp}[cf. {\cite[Proposition 3.4]{Kaimanovich05a}}] \label{prp:mf}
There is a canonical one-to-one correspondence between equivariant Markov chains on a measurable groupoid~$\Gb$ and measurable target fibred systems of probability measures on $\Gb$.
\end{prp}

\begin{proof}
The restriction of the system of the transition probabilities $\{\pi^\gb\}_{\gb\in \Gb}$ of an equivariant Markov chain to the space of objects $X$ (which is embedded into~$\Gb$ by the unit inclusion map $\eb$) produces a system of probability measures
\begin{equation} \label{eq:thx}
\th^x = \pi^{\eb_x}
\end{equation}
which is target fibred by condition \eqref{eq:tg}, and the restriction of a measurable system~$\{\pi^g\}$ from $\Gb$ to $X$ is a measurable system as well.

Conversely, if $\{\th^x\}_{x\in X}$ is a measurable target fibred system of probability measures, then put
\begin{equation} \label{eq:gx}
\pi^\gb = \gb \, \th^{\uun{\gb}} \qquad\forall\,\gb\in \Gb \;.
\end{equation}
Since $\th^{\uun{\gb}}$ is concentrated on $\Gb^{\,\uun{\gb}}$, the measures $\pi^\gb$ are well-defined, satisfy conditions \eqref{eq:tg} and \eqref{eq:g1g2}, and the system $\{\pi^\gb\}$ is measurable.
\end{proof}

\begin{rem}
The above correspondence clearly preserves the absolute continuity. Namely, the transition probabilities $\{\pi^\gb\}$ of an equivariant Markov chain on a groupoid $\Gb$ are \textsf{absolutely continuous} with respect to a Haar system $\lab=\{\la^x\}$ in the sense that
$$
\pi^\gb \prec \la^{\ov{\gb}} \qquad \forall\,\gb \in \Gb
$$
if and only if $\th^x\prec \la^x$ for all measures from the corresponding target fibred system~$\{\th^x\}$.
\end{rem}

\begin{rem}
The product of two equivariant Markov operators on a groupoid is easily seen to be equivariant as well, which, in view of the above correspondence, gives rise to an associative binary operation on the set of measurable target fibred systems of probability measures. This operation can be further extended to the set of all measurable target fibred systems of (not necessarily probability and not necessarily positive) measures under appropriate finiteness conditions which guarantee that the composition be well-defined. For groups (when there is only one object, so that any target fibred system consists just of a single measure) this is precisely the usual \textsf{convolution}, and it is natural to use this term for groupoids as well.

In the case of absolutely continuous systems of measures the convolution can then be considered as an operation with the corresponding densities, and this is how the convolution of functions on a groupoid is usually defined (e.g., see \mbox{Renault} \cite[Section~II.1]{Renault80} or Anantharaman-Delaroche\,--\,Renault \cite[Section 3.1.a]{Anantharaman-Renault00}). However, we emphasize that, in precisely the same way as for groups, the natural domain of the convolution operation on groupoids consists of (systems~of) \emph{measures} rather than of \emph{functions} (which only arise as densities with respect to certain reference measures, e.g., with respect to a Haar system).
\end{rem}

\subsection{Asymptotically invariant sequences of transition operators}

The discrepancy functions \eqref{eq:fIDR} that appear in the definition of amenability of a groupoid (\dfnref{dfn:reiter}) admit a natural interpretation in terms of the transition probabilities~$\{\pi^\gb\}$ of the equivariant Markov chain determined by a target fibred system~$\{\th^x\}$. Indeed,
by formulas \eqref{eq:thx} and \eqref{eq:gx}, respectively, for any $\gb\in\Gb$
$$
\th^{\ov \gb}=\pi^{\eb_{\ov \gb}}=\pi^{\ov \gb}
$$
(under the usual identification of objects $x\in X$ with the corresponding identity morphisms~$\eb_x$), and
$$
\gb \, \th^{\uun{\gb}}=\pi^\gb \;,
$$
or, in the operator notation \eqref{eq:op},
$$
\th^{\ov \gb}=\de_{\ov \gb} P \;, \qquad \gb \,\th^{\uun{\gb}}=\de_\gb P \;.
$$

As we have already mentioned in \remref{rem:fibre}, the sample paths of an equivariant Markov chain are confined to the fibres $\Gb^{\,x}$ of the target map; each of these fibres contains a distinguished point (the identity morphism $\eb_x$ at $x$), and the value of function \eqref{eq:fIDR} at a point $\gb\in \Gb$ is therefore nothing but the total variation distance between the transitions probabilities issued from the point $\gb$ and from the distinguished point $\ov \gb$ of its target fibre~$\Gb^{\,\ov \gb}$.

\begin{dfn} \label{dfn:dscrp}
The \textsf{discrepancy} of an equivariant Markov chain (or, of the corresponding transition operator $P$) on a groupoid $\Gb$ at a point $\gb\in \Gb$ is the total variation
$$
\De(\gb,P) = \left\| \pi^\gb - \pi^{\ov \gb} \right\| = \left\| (\de_\gb - \de_{\ov \gb}) P \right\| \;.
$$
The \textsf{mean discrepancy} with respect to a probability measure $m$ on $\Gb$ is the \mbox{$m$-average}
$$
\De(m,P) = \int \De(\gb,P)\,dm(\gb)
$$
of the \textsf{discrepancy function} $\De(\cdot,P)$.
\end{dfn}

\dfnref{dfn:IDRop} and \prpref{prp:am} below are the reformulations of the definitions of the strongly asymptotically invariant sequences and of amenability from \secref{sec:amen} in terms of the discrepancy functions of equivariant Markov chains obtained by replacing, with the help of \prpref{prp:mf}, target fibred systems of measures with the associated equivariant transition operators.

\begin{dfn} \label{dfn:IDRop}
A sequence of equivariant transition operators on a measured groupoid $(\Gb,\lab,\ka)$ satisfies the \textsf{integral strong asymptotic invariance (ISAI) condition} if the sequence of their discrepancy functions converges to 0 in the weak$^*$ topology of the space $L^\infty(\Gb,\lab\star\ka)$.
\end{dfn}

\begin{rem} [cf. \remref{rem:mw}] \label{rem:conv}
The weak$^*$ convergence
$$
\De(\cdot,P_n)\to 0
$$
in the above definition is equivalent to the convergence
$$
\De(\mf,P_n)\to 0
$$
just for a single reference probability measure $\mf\sim\la\star\ka$.
\end{rem}

\begin{rem} [cf. \remref{rem:pw}] \label{rem:pw1}
The integral strong asymptotic invariance condition follows from the \textsf{pointwise strong asymptotic invariance condition}, which in the Markov setup amounts to the almost everywhere convergence of the discrepancy functions
$$
\De(g,P_n) \to 0 \qquad \text{for}\;(\lab\star\ka)\text{-a.e.} \;\gb\in \Gb \;.
$$
\end{rem}

\begin{prp} \label{prp:am}
A measured groupoid is amenable if and only if it admits an ISAI sequence of equivariant Markov chains with absolutely continuous transition probabilities.
\end{prp}

Our main technical result (to be proved in \secref{sec:proof} below) is

\begin{thm} \label{thm:IDRpower}
If $(P_n)$ is an ISAI sequence of equivariant transition operators with absolutely continuous transition probabilities on a measured groupoid, then there exists an infinite convex combination $P$ of the operators $P_n$ such that the sequence of powers $P^n$ is also ISAI.
\end{thm}

\thmref{thm:IDRpower} then implies the following characterization of amenability of a measured groupoid:

\pagebreak

\begin{thm} \label{thm:powers}
The following conditions on a measured groupoid $(\Gb,\lab,\ka)$ are equivalent:
\begin{enumerate}[{\rm (i)}]
\item
$(\Gb,\lab,\ka)$ is amenable;
\item
there exists a sequence $(\thb_n)$ of measurable target fibred systems of absolutely continuous probability measures which satisfies the pointwise strong asymptotic invariance condition
$$
\left\| \gb \, \theta_{n}^{\uun{\gb}} - \theta_{n}^{\ov{\gb}} \right\| \xrightarrow[n \to \infty]{} 0 \qquad
\text{for}\;(\lab\star\ka)\text{-a.e.} \;\gb\in \Gb \;,
$$
or, in terms of the corresponding transition operators $P_n$,
$$
\left\| (\de_\gb - \de_{\ov{\gb}}) P_n \right\| \xrightarrow[n \to \infty]{} 0 \qquad
\text{for}\;(\lab\star\ka)\text{-a.e.} \;\gb\in \Gb \;;
$$
\item
the systems $\thb_n$ from condition \textup{(ii)} can be chosen to be the convolution powers of a single such system, i.e., the corresponding transition operators are the powers $P_n=P^n$ of a single operator $P$.
\end{enumerate}
\end{thm}

\begin{proof}
Since the implication (iii)$\implies$(ii) is trivial, whereas the implication \mbox{(ii)$\implies$(i)} follows from the fact that the pointwise strong asymptotic invariance condition is stronger than the integral one (see \remref{rem:pw} and \remref{rem:pw1}), we only have to verify the implication (i)$\implies$(iii), which is an immediate consequence of \prpref{prp:am} and \thmref{thm:IDRpower}.

Indeed, the sequence of the discrepancy functions $\De(\cdot,P^n)$ is uniformly bounded and monotone in $n$, see inequality \eqref{eq:PQ} below, and therefore the integral and the pointwise strong asymptotic invariance conditions are equivalent for any power sequence~$(P^n)$.
\end{proof}

\subsection{Inequalities for the discrepancy function}

Here we establish certain auxiliary properties of discrepancy functions and ISAI sequences used in the proof of \thmref{thm:IDRpower}.

It is straightforward that for any pair of equivariant transition operators $P,Q$ the discrepancy function of the product $PQ$ satisfies the inequality
\begin{equation} \label{eq:PQ}
\begin{aligned}
\De(\gb,PQ) &= \left\| (\de_\gb - \de_{\ov \gb}) PQ \right\| \\
&\le \left\| (\de_\gb - \de_{\ov \gb}) P \right\| = \De(\gb,P)
\end{aligned}
\end{equation}
for any $\gb\in \Gb$. In what concerns the product in the opposite order,
\begin{equation} \label{eq:QP}
\begin{aligned}
\De(\gb,QP) &= \left\| (\de_\gb - \de_{\ov \gb}) QP \right\| \\
&\le \left\| \de_\gb QP - \de_{\ov \gb} P \right\| + \left\| \de_{\ov \gb} QP - \de_{\ov \gb} P \right\| \\
&\le \De(\de_\gb Q,P) + \De(\de_{\ov \gb}Q,P) \;,
\end{aligned}
\end{equation}
where we have used the fact that if a probability measure $m$ is concentrated on a fibre~$\Gb^{\,x}$, then
$$
\begin{aligned}
\left\| \left( m-\de_x \right) P \right\|
&= \left\| \int ( \de_\gb - \de_x) P \,dm(\gb) \right\| \\
&\le \int \| ( \de_\gb - \de_x ) P \| \,dm(\gb)
= \De(m,P) \;.
\end{aligned}
$$

By integrating inequalities \eqref{eq:PQ} and \eqref{eq:QP} we then obtain

\pagebreak

\begin{lem} \label{lem:ineq}
For any probability measure $m$ on $\Gb$ and any pair of equivariant transition operators $P,Q$
\begin{enumerate}[{\rm (i)}]
\item $\De(m,PQ)\le\De(m,P)\;,$
\item $\De(m,QP)\le\De(mQ,P)+\De(\ov m Q,P)\;,$
\end{enumerate}
where $\ov m$ denotes the image of the measure $m$ under the target map $\ta:\gb\mapsto\ov \gb$ composed (as usual) with the unit inclusion $\eb:X\to\Gb$.
\end{lem}

If $Q$ is an equivariant transition operator with absolutely continuous transition probabilities on a measured groupoid $(\Gb,\lab,\ka)$, and $m$ is a measure absolutely continuous with respect to $\lab\star\ka$, then the measures $mQ$ and $\ov m Q$ also have this property, and therefore \lemref{lem:ineq}(ii) implies

\begin{prp} \label{prp:QPn}
If $(P_n)$ is an ISAI sequence of transition operators on a measured groupoid, then for any other equivariant transition operator with absolutely continuous transition probabilities $Q$ the sequence $(QP_n)$ is also ISAI.
\end{prp}

\subsection{ISAI power sequences (proof of \thmref{thm:IDRpower})} \label{sec:proof}

Let us first fix a sequence of coefficients $t_i > 0$ with
$$
\sum_{i=1}^\infty t_i=1 \;,
$$
and choose an increasing integer sequence $k_i$ with
\begin{equation} \label{eq:ki}
\left(t_1 + \dots + t_{i-1} \right)^{k_i} = \ep_i \searrow 0 \;.
\end{equation}
By using \prpref{prp:QPn} we can inductively define a sequence of operators
$$
R_i=P_{n_i}
$$
by starting with $n_1=1$, and then for $i>1$ putting $n_i$ to be the minimal $n>n_{i-1}$ such that
\begin{equation} \label{eq:QPn}
\De(\mf,QP_n) \le \ep_i
\end{equation}
for any operator $Q$ which is a product of not more than $k_i$ of the operators $R_1,\dots,R_{i-1}$ (taken in any order and counted with their multiplicities), where $\mf$ is a fixed probability measure on $\Gb$ equivalent to $\lab\star\ka$. We claim that the operator
$$
P = \sum_{i=1}^\infty t_i R_i = \sum_{i=1}^\infty t_i P_{n_i}
$$
does the job. By \remref{rem:conv} we have to verify just that
$$
\De(\mf,P^k)\xrightarrow[k \to \infty]{} 0 \;.
$$

Indeed, for any integer $k$
\begin{equation} \label{eq:Pk}
P^k = \sum_I t_I R_I \;,
\end{equation}
where the sum is taken over all multi-indices $I=(i_1,i_2,\dots, i_k)$ with the entries $i_j\ge 1$, and
$$
t_I = t_{i_1} t_{i_2} \dots t_{i_k} \;, \qquad R_I = R_{i_1} R_{i_2} \dots R_{i_k} \;.
$$
Now, for $k=k_i$ we can split the sum \eqref{eq:Pk} as
$$
P^k = \sum_{I:|I|<i} t_I R_I + \sum_{I:|I|\ge i} t_I R_I \;,
$$
where
$$
|I| = \max_j i_j
$$
denotes the maximum of a multi-index $I=(i_1,i_2,\dots, i_k)$. By \eqref{eq:ki}
$$
\sum_{I:|I|<i} t_I = \ep_i \;,
$$
whereas by \eqref{eq:QPn} and \lemref{lem:ineq}(i)
$$
\De(\mf, R_I) \le \ep_i \qquad \forall\,I: |I|\ge i \;,
$$
whence
$$
\begin{aligned}
\De(\mf,P^k)
&\le \sum_{I:|I|<i} t_I \De(\mf,R_I) + \sum_{I:|I|\ge i} t_I \De(\mf,R_I) \\
&\le 2 \sum_{I:|I|<i} t_I + \ep_i \sum_{I:|I|\ge i} t_I
\le 3\ep_i \;.
\end{aligned}
$$
Since the sequence $\De(\mf,P^k)$ is non-increasing in $k$ by \lemref{lem:ineq}(i), the claim follows.

\subsection{0--2 laws for Markov chains} \label{sec:02}

In order to interpret \thmref{thm:powers} in entirely qualitative probabilistic terms, let us first recall the corresponding background definitions and results in the general situation. Given a transition operator~$P$ and an initial measure~$m$ on a state space $\Om$, we denote by $\Pb_m$ the associated measure on the path space~$\Om^{\ZZ_+}$. Note that $m$ (and, therefore, $\Pb_m$ as well) need not be normalized or finite.

There are two complete $\si$-algebras on the path space $\bigl(\Om^{\ZZ_+},\Pb_m\bigr)$ which describe the behaviour of the chain at infinity: the \textsf{tail $\si$-algebra} $\Af^\infty$ and its subalgebra
$\Ef\subset\Af^\infty$ (the \textsf{exit $\si$-algebra}) which consists of the time shift invariant events. These $\si$-algebras can also be described as the measurable envelopes of the respective \textsf{synchronous and asynchronous tail equivalence relations} on the path space:
$$
(\om_0,\om_1,\dots) \approx (\om'_0,\om'_1,\dots) \iff \exists\, n\ge 0: \om_{n+i}=\om'_{n+i} \;\;\forall\,i\ge 0 \;,
$$
and
$$
\quad\;\; (\om_0,\om_1,\dots) \sim (\om'_0,\om'_1,\dots) \iff \exists\, n,n'\ge 0: \om_{n+i}=\om'_{n'+i} \;\;\forall\,i\ge 0 \;.
$$
The associated quotient spaces of the path space are called, respectively, the \textsf{tail boundary} and the \textsf{Poisson boundary} of the Markov chain. Since the exit $\si$-algebra is a subalgebra of the tail one, the Poisson boundary is a quotient of the tail boundary.

The Poisson boundary is responsible, via the Poisson formula, for an integral representation of bounded \textsf{harmonic functions} on the state space (i.e., such functions~$f$ that $Pf=f$), whereas the tail boundary represents the bounded \textsf{space-time harmonic functions} (i.e., the sequences $(f_n)$ of functions on the state space with $f_n=P f_{n+1}$). In particular, the absence of non-constant bounded harmonic functions (the \textsf{Liouville property}) is equivalent to the triviality of Poisson boundary, i.e., to the triviality $\Pb_m$-(mod 0) of the exit $\si$-algebra.

Necessary and sufficient conditions for the triviality of the tail and the exit $\si$-algebras are provided by the respective \emph{0-2 laws} \cite{Derriennic76}, \cite{Kaimanovich92}. Below we shall only need the ``0~part'' of these laws in the situation when the measure $m$ is \textsf{quasi-substationary}, i.e., when the measure class of the step 1 distribution~$mP$ is dominated by that of $m$; in particular, this is the case when $m$-almost all transition probabilities are absolutely continuous with respect to~$m$.

\pagebreak

\begin{prp}[Derriennic {\cite[Theorem 6]{Derriennic76}}] \label{prp:02}
If an initial distribution~$m$ of the Markov chain with a transition operator $P$ is quasi-substationary, then
\begin{enumerate}[{\rm (i)}]
\item
the tail $\si$-algebra $\Af^\infty$ is trivial $\Pb_m$\textup{-(mod 0)} if and only if
$$
\left\| (\al-\be)P^n \right\| \to 0
$$
for any two probability measures $\al,\be\prec m$;
\item
the exit $\si$-algebra $\Ef$ is trivial $\Pb_m$\textup{-(mod 0)} if and only if
$$
\left\| (\al-\be)Q_n \right\| \to 0
$$
for any two probability measures $\al,\be\prec m$, where
$$
Q_n = \frac1n \left( P + P^2 + \dots + P^n \right)
$$
are the Ces{\`a}ro averages of the powers of the operator $P$.
\end{enumerate}
\end{prp}

\begin{rem}
If the measure $m$ is \textsf{stationary} (i.e., $mP=m$), then the associated measure~$\Pb_m$ on the path space is invariant with respect to the time shift $T$, and conditions (i) and (ii) above are equivalent to mixing and ergodicity of $T$, respectively; this is how these conditions are sometimes referred to (e.g., see Rosenblatt \cite{Rosenblatt81}). The simplest example in which (ii) is satisfied, whereas (i) is not (ergodicity without mixing), is provided by the deterministic random walk on the order~2 cyclic group $\ZZ_2$, or, in a less degenerate form, by the simple random walk on the infinite cyclic group $\ZZ$ (in both cases one takes for $m$ the counting measure on the state space).

Notice that even if the tail $\si$-algebra is non-trivial $\Pb_m$-(mod~0), it may still be trivial with respect to other ``smaller'' initial distributions (for instance, if a random walk on a group is Liouville, then the tail $\si$-algebra is trivial with respect to any one-point initial distribution), see \cite{Kaimanovich-Vershik83}, \cite{Kaimanovich92} for more details.
\end{rem}

\begin{rem} \label{rem:tail}
The sequence of Ces\`aro averages in condition (ii) can be replaced with the sequence of the averages of the powers of $P$ along any approximate invariant mean on $\ZZ_+$ \cite[Theorem 2.3]{Kaimanovich92}, in particular, with the sequence of the powers $R^n$ of the average
$$
R = \sum_k p_k P^k
$$
for any probability distribution $(p_k)$ on $\ZZ_+$ with the greatest common divisor 1, for instance, one can take
$$
R=\tfrac12 (P+P^2) \;.
$$
Thus, under the conditions of \prpref{prp:02} the exit $\si$-algebra of the operator $P$, and both the tail and the exit $\si$-algebras of the operator~$R$ are all either trivial or non-trivial $\Pb_m$-(mod~0) simultaneously (actually, these three $\si$-algebras can be identified in a natural sense, see \cite{Derriennic76}, \cite{Kaimanovich92}).
\end{rem}

\subsection{Asymptotic invariance and the fibrewise Liouville property}

We can now return to the equivariant Markov chains on groupoids. Since the sample paths of equivariant chains are confined to the target fibres $\Gb^{\,x}, x\in X$ (see \remref{rem:fibre}), any equivariant transition operator~$P$ can be considered as a collection of the \textsf{fibrewise transition operators}~$P^x$.

\begin{dfn}[{\cite[Definition 4.1]{Kaimanovich05a}}]
The tail (resp., Poisson) boundary of an equivariant Markov chain on a measured groupoid $(\Gb,\lab,\ka)$ is \textsf{fibrewise trivial} if the tail (resp., Poisson) boundary of $\ka$-almost every fibrewise chain is trivial with respect to the corresponding initial measure~$\la^x, x\in X$. If the Poisson boundary is fibrewise trivial, i.e., if almost every fibrewise chain is Liouville, we also say that the chain is \textsf{fibrewise Liouville}.
\end{dfn}

\prpref{prp:02}(i) immediately implies

\begin{prp} \label{prp:tdr}
The tail boundary of an equivariant Markov chain with absolutely continuous transition probabilities on a measured groupoid is fibrewise trivial if and only if the sequence of the powers of the associated transition operator satisfies the pointwise ($\equiv$~integral) \footnotemark\ strong asymptotic invariance condition.
\end{prp}

\footnotetext{\,See the proof of \thmref{thm:powers} for the equivalence of the pointwise and the integral asymptotic invariance conditions in the case of power sequences.}

Therefore, \thmref{thm:powers} leads to

\begin{thm} \label{thm:main}
The following conditions on a measured groupoid $(\Gb,\lab,\ka)$ are equivalent:
\begin{enumerate}[{\rm (i)}]
\item
the groupoid is amenable;
\item
the groupoid admits an equivariant Markov chain with absolutely continuous transition probabilities such that its tail boundary is fibrewise trivial;
\item
the groupoid admits a fibrewise Liouville equivariant Markov chain with absolutely continuous transition probabilities.
\end{enumerate}
\end{thm}

\begin{proof}
(i)$\implies$(ii)
This is a combination of \thmref{thm:powers} and \prpref{prp:tdr}.

(ii)$\implies$(iii)
This is obvious, because the fibrewise Poisson boundaries are quotients of the respective tail boundaries.

(iii)$\implies$(i)  \cite[Theorem 4.2]{Kaimanovich05a}
By the absolute continuity assumption the transition probability $\pi^x=\de_x P$ is absolutely continuous with respect to the corresponding Haar measure $\la^x$ for $\ka$-almost every $x\in X$, and, having fixed such an $x$, for $\la^x$-almost every $\gb\in \Gb^{\,x}$ the transition probability $\pi^\gb=\de_\gb P$ is also absolutely continuous with respect to~$\la^x$. Therefore, by \prpref{prp:02}(ii)
$$
\left\| (\de_\gb - \de_{\ov \gb}) PQ_n \right\| \to 0
$$
for $\lab\star\ka$-almost every $\gb\in \Gb$, i.e., the sequence $(PQ_n)$ satisfies the pointwise asymptotic invariance condition, whence $G$ is amenable by \prpref{prp:am}.
\end{proof}

\begin{rem}
The implication (iii)$\implies$(ii) is also very easy to verify directly because the Liouville property for the Markov chain with a transition operator $P$ is equivalent to the triviality of the tail boundary for the convex combination $\frac12(P+P^2)$, see \remref{rem:tail}.
\end{rem}

\section{Amenable actions} \label{sec:3}

\subsection{Action groupoid}

If $G:X\acts$ is a (left) group action, then the action space $X$ can be considered as the set of objects (``points'') of the associated \textsf{action groupoid} $\Gb$. Its morphisms are the ``arrows''
$$
\gb = \left( y \xleftarrow[]{\;g\;} x  \right)
$$
joining the \textsf{source} and the \textsf{target}
$$
\un\gb=x\in X \;, \qquad \ov\gb=y=gx\in X \;,
$$
respectively, and marked with the corresponding group element (\textsf{label})
$$
\lb{\gb}=g\in G \;,
$$
so that
$$
\lbb{\gb\gb\myp}=\lb{\gb}\cdot\lbb{\gb\myp}
$$
whenever the composition $\gb\gb\myp$ is well-defined (cf. \figref{fig:groupoid}).

In this situation the source of any morphism $\gb$ is uniquely determined by the target $\ov\gb\in X$ and the label $\lb\gb\in G$, so that one can identify the groupoid $\Gb$ with the product $G\times X$ by using the correspondence
\begin{equation} \label{eq:gbx}
\Gb\cong G\times X\;, \qquad  \gb \cong \left(\lb\gb,\ov\gb\right) \;,
\end{equation}
and under this identification
$$
\un{(g,x)} = g^{-1}x \;, \qquad \ov{(g,x)}=x \;, \qquad \lbb{(g,x)} = g \;.
$$
In particular, the target map $\gb\mapsto\ov\gb$ takes the form of the projection
$$
(g,x)\mapsto\ov{(g,x)}=x \;,
$$
which provides an identification of the fibres $\Gb^{\,x}$ of the target map with the group~$G$.

The composability condition for two pairs $(g,x)$ and $(g',x')$ is $x'=g^{-1}x$, and their composition is
\begin{equation} \label{eq:ggx}
(g,x) \left(g',g^{-1}x\right) = (gg', x) \;.
\end{equation}
Therefore,
$$
(g,x)\,\Gb^{\,g^{-1}x} = \,\Gb^{\,x} \;,
$$
and the left composition with a morphism $(g,x)$ on the fibre $\Gb^{\,g^{-1}x}\cong G$ amounts to multiplying the corresponding group elements on the left by $g$ and reattaching this copy of the group $G$ to the point $x$ instead of $g^{-1}x$.

\subsection{Actions with a quasi-invariant measure}

Let us now assume that $G$ is a locally compact second countable topological group, and $\la$ is a left Haar measure on it. The measure $\la$ then gives rise to the Haar system $\lab=\{\la^x\}$ on the action groupoid $\Gb$ obtained by endowing each fibre $\Gb^{\,x}\cong G$ with the measure $\la^x\cong\la$.

Further, let us assume that the action space $X$ is endowed with a measure $\ka$. Then the measure $\lab\star\ka$ on $\Gb\cong G\times X$ is nothing but the product measure $\la\otimes\ka$, and the quasi-invariance of the measure $\ka$ with respect to the action of the group $G$ means precisely that
$(\Gb,\lab,\ka)$ is a measured groupoid in the sense of the definition from \secref{sec:am}.

In view of the identification $\Gb^{\,x}\cong G$, a measurable target fibred system $\thb$ of probability measures on the action groupoid $\Gb$ is just a measurable map
$$
\thb: x\mapsto\th^x
$$
from the action space $X$ to the space $\Pc(G)$ of probability measures on the group $G$, or, for an absolutely continuous system, to the space $\Pca(G)$ of probability measures on $G$ absolutely continuous with respect to the Haar measure $\la$.

The discrepancy function \eqref{eq:fIDR} from the definition of amenability then takes the form
\begin{equation} \label{eq:df}
\gb \mapsto
\left\| \gb\,\theta^{\uun{\gb}} - \theta^{\ov{\gb}} \right\|
= \bigl\| g\,\theta^{g^{-1}x} - \theta^x \bigr\| \;, \qquad \gb=(g,x)\in\Gb \;.
\end{equation}
Therefore, \dfnref{dfn:reiter} applied to the action groupoid $(\Gb,\lab,\ka)$ means that the action of the group $G$ on the space $(X,\ka)$ is amenable if there exists a sequence of measurable maps
$$
\thb_n: x\mapsto\th_n^x \;, \qquad X \to \Pca(G) \;,
$$
which are \textsf{asymptotically equivariant} in the sense that the sequence of their discrepancy functions \eqref{eq:df} converges to 0 in the topology $\si(L^\infty,L^1)$ of the space $L^\infty(G\times X,\la\otimes\ka)$ (the integral strong asymptotic invariance condition).

\begin{rem}
By a change of variables the discrepancy functions \eqref{eq:df} in the above definition can be replaced with the functions
$$
(g,x) \mapsto \bigl\| g\,\theta^{x} - \theta^{gx} \bigr\| \;.
$$
\end{rem}

\subsection{Random walks in random environment}

For describing the ``local'' Markov chains on the fibres $\Gb^{\,x}\cong G$ arising from a ``global'' equivariant Markov chain on the action groupoid (cf. \remref{rem:fibre}) it is convenient to adopt the terminology of random walks in random environment.

The presence of a group structure on the state space of a Markov chain allows one to talk about the increments between consecutive states and to describe the chain in terms of their distributions. If $\{\pi^g\}_{g\in G}$ is a family of transition probabilities on a group $G$, then the \textsf{increment distribution} at a point $g\in G$ is the translate
$$
\mu^g = g^{-1}\pi^g \;,
$$
so that the Markov transition from $g$ consists in the (right) multiplication
$$
g\xrsquigarrow{h\sim\mu^g} gh
$$
by a $\mu^g$-distributed random increment $h$.

The \textsf{environment} (collection of the increment distributions) $\mub=\{\mu^g\}_{g\in G}$ can be considered as a $\Pc(G)$-valued configuration (or a field, in the continuous case) on the group $G$. The usual (right) random walk on $G$ with the step distribution~$\mu$ corresponds then to the constant environment $\mu^g\equiv\mu$. In the presence of a probability measure $\nu$ on the \textsf{space of environments} $\Ec(G)$ on a group $G$ one talks about the associated \textsf{random walk in random environment (RWRE)}, which amounts to first sampling an environment $\mub$ from the distribution $\nu$ and then performing the random walk on $G$ with the transition probabilities $\pi^g=g\mu^g$ determined by the environment $\mub$ \cite{Kaimanovich90a}. The group~$G$ naturally acts on the space of environments $\Ec(G)$ as
$$
(g\mub)^{g'} = \mu^{g^{-1}g'} \;, \qquad g,g'\in G\;,
$$
and the measure $\nu$ is usually assumed to be invariant or quasi-invariant with respect to this action.

Now, any measurable map $\thb: X\to\Pc(G)$ determines, by formulas \eqref{eq:gx} and~\eqref{eq:ggx}, the corresponding equivariant Markov chain on the action groupoid $\Gb$ with the transition probabilities
$$
\pi^{(g,x)} = g \, \th^{g^{-1}x}
$$
on the fibres $\Gb^{\,x}\cong G$. In other words, the restriction of the chain to each fibre $\Gb^{\,x}$ is the random walk in the environment $\mub=\env_\thb(x)$ defined by the map
\begin{equation} \label{eq:env}
\env_\thb: x\mapsto\mub=\{\mu^g\} \;, \qquad \mu^g = \th^{g^{-1}x} \;,
\end{equation}
from the action space $X$ to the space of environments $\Ec(G)$. This map is clearly $G$-equivariant, so that the quasi-invariance of the measure $\nu$ implies the quasi-invariance of its image $\env_\thb\nu$ on the space of environments.

We can now restate \thmref{thm:main} in the case of group actions in terms of random walks in random environment.

\begin{thm} \label{thm:actions}
Let $G$ be a second countable locally compact group endowed with a Haar measure $\la$. Then the following conditions on a measure class preserving action of the group $G$ on a measure space $(X,\ka)$ are equivalent:
\begin{enumerate}[{\rm (i)}]
\item
the action is amenable;
\item
there exists a measurable map $\thb:X\mapsto\Pca(G)$ such that the tail boundary of the random walk on the group $G$ in the associated random environment \eqref{eq:env} is almost surely $\Pb_\la$\textup{-(mod 0)} trivial;
\item
there exists a measurable map $\thb:X\mapsto\Pca(G)$ such that the random walk on the group $G$ in the associated random environment \eqref{eq:env} is almost surely Liouville;
\end{enumerate}
\end{thm}

\bibliographystyle{amsalpha}
\bibliography{buehler-kaimanovich.bbl}

\providecommand{\bysame}{\leavevmode\hbox to3em{\hrulefill}\thinspace}
\providecommand{\MR}{\relax\ifhmode\unskip\space\fi MR }
\providecommand{\MRhref}[2]{%
  \href{http://www.ams.org/mathscinet-getitem?mr=#1}{#2}
}
\providecommand{\href}[2]{#2}
\begin{thebibliography}{CFW81}

\bibitem[ADR00]{Anantharaman-Renault00}
Claire Anantharaman-Delaroche and Jean Renault, \emph{Amenable groupoids},
  Monographies de L'Enseignement Math\'ematique [Monographs of L'Enseignement
  Math\'ematique], vol.~36, L'Enseignement Math\'ematique, Geneva, 2000, With a
  foreword by Georges Skandalis and Appendix B by E. Germain. \MR{2001m:22005}

\bibitem[ADR01]{Anantharaman-Renault01}
\bysame, \emph{Amenable groupoids}, Groupoids in analysis, geometry, and
  physics ({B}oulder, {CO}, 1999), Contemp. Math., vol. 282, Amer. Math. Soc.,
  Providence, RI, 2001, pp.~35--46. \MR{1855241}

\bibitem[AEG94]{Adams-Elliott-Giordano94}
Scot Adams, George~A. Elliott, and Thierry Giordano, \emph{Amenable actions of
  groups}, Trans. Amer. Math. Soc. \textbf{344} (1994), no.~2, 803--822.
  \MR{94k:22010}

\bibitem[Ano94]{Anosov94}
D.~V. Anosov, \emph{On {N}. {N}. {B}ogolyubov's contribution to the theory of
  dynamical systems}, Uspekhi Mat. Nauk \textbf{49} (1994), no.~5(299), 5--20.
  \MR{1311227 (96a:01029)}

\bibitem[Aze69]{Azencott69}
Robert Azencott, \emph{Espaces de {P}oisson des groupes de {L}ie}, C. R. Acad.
  Sci. Paris S\'{e}r. A-B \textbf{268} (1969), A139--A142. \MR{241573}

\bibitem[Aze70]{Azencott70}
\bysame, \emph{Espaces de {P}oisson des groupes localement compacts},
  Springer-Verlag, Berlin, 1970, Lecture Notes in Mathematics, Vol. 148. \MR{58
  \#18748}

\bibitem[Bar18]{Bartholdi18}
Laurent Bartholdi, \emph{Amenability of groups and {$G$}-sets}, Sequences,
  groups, and number theory, Trends Math., Birkh\"{a}user/Springer, Cham, 2018,
  pp.~433--544. \MR{3799933}

\bibitem[Bog39]{Bogolyubov39}
N.~N. Bogolyubov, \emph{On some ergodic properties of continuous transformation
  groups}, Nauch. Zap. Kiev Univ. Phys.-Mat. Sb. \textbf{4} (1939), no.~5,
  45--52 (in Ukrainian), also: \emph{Selected works in mathematics}, Fizmatlit,
  Moscow, 2006, pp. 213--222 (in Russian).

\bibitem[Bow77]{Bowen77}
Rufus Bowen, \emph{Anosov foliations are hyperfinite}, Ann. of Math. (2)
  \textbf{106} (1977), no.~3, 549--565. \MR{57 \#1569}

\bibitem[B{\"u}h06]{Buhler06p}
Theo B{\"u}hler, \emph{On a conjecture of {V}adim {K}aimanovich}, preprint,
  2006.

\bibitem[CFW81]{Connes-Feldman-Weiss81}
A.~Connes, J.~Feldman, and B.~Weiss, \emph{An amenable equivalence relation is
  generated by a single transformation}, Ergodic Theory Dynamical Systems
  \textbf{1} (1981), no.~4, 431--450 (1982). \MR{84h:46090}

\bibitem[CL18]{Chu-Li18}
Cho-Ho Chu and Xin Li, \emph{Amenability, {R}eiter's condition and {L}iouville
  property}, J. Funct. Anal. \textbf{274} (2018), no.~12, 3291--3324.
  \MR{3787593}

\bibitem[Com77]{Combes77}
Fran{\c c}ois Combes, \emph{Review of \emph{Hyperfinite factors and amenable
  ergodic actions}, by {R}obert {J}.\ {Z}immer}, MathSciNet MR0470692, 1977.

\bibitem[Con94]{Connes94}
Alain Connes, \emph{Noncommutative geometry}, Academic Press, Inc., San Diego,
  CA, 1994. \MR{1303779}

\bibitem[CW89]{Connes-Woods89}
A.~Connes and E.~J. Woods, \emph{Hyperfinite von {N}eumann algebras and
  {P}oisson boundaries of time dependent random walks}, Pacific J. Math.
  \textbf{137} (1989), no.~2, 225--243. \MR{90h:46100}

\bibitem[Day49]{Day49}
Mahlon~M. Day, \emph{Means on semigroups and groups}, Bull. Amer. Math. Soc.
  \textbf{55} (1949), 1054--1055, abstract 55-11-507.

\bibitem[Day50]{Day50}
\bysame, \emph{Means for the bounded functions and ergodicity of the bounded
  representations of semi-groups}, Trans. Amer. Math. Soc. \textbf{69} (1950),
  276--291. \MR{0044031 (13,357e)}

\bibitem[Day57]{Day57}
\bysame, \emph{Amenable semigroups}, Illinois J. Math. \textbf{1} (1957),
  509--544. \MR{0092128 (19,1067c)}

\bibitem[Day61]{Day61}
\bysame, \emph{Fixed-point theorems for compact convex sets}, Illinois J. Math.
  \textbf{5} (1961), 585--590. \MR{0138100}

\bibitem[Der76]{Derriennic76}
Yves Derriennic, \emph{Lois ``z\'ero ou deux'' pour les processus de {M}arkov.
  {A}pplications aux marches al\'eatoires}, Ann. Inst. H. Poincar\'e Sect. B
  (N.S.) \textbf{12} (1976), no.~2, 111--129. \MR{54 \#11508}

\bibitem[Die60]{Dieudonne60}
Jean Dieudonn\'e, \emph{Sur le produit de composition. {II}}, J. Math. Pures
  Appl. (9) \textbf{39} (1960), 275--292. \MR{0115100}

\bibitem[Dix50]{Dixmier50}
Jacques Dixmier, \emph{Les moyennes invariantes dans les semi-groupes et leurs
  applications}, Acta Sci. Math. (Szeged) \textbf{12} (1950), 213--227.
  \MR{37470}

\bibitem[DS88]{Dunford-Schwartz88}
Nelson Dunford and Jacob~T. Schwartz, \emph{Linear operators. {P}art {I}},
  Wiley Classics Library, John Wiley \& Sons, Inc., New York, 1988, General
  theory, With the assistance of William G. Bade and Robert G. Bartle, Reprint
  of the 1958 original, A Wiley-Interscience Publication. \MR{1009162}

\bibitem[EG93]{Elliott-Giordano93}
G.~A. Elliott and T.~Giordano, \emph{Amenable actions of discrete groups},
  Ergodic Theory Dynam. Systems \textbf{13} (1993), no.~2, 289--318.
  \MR{94i:22023}

\bibitem[FK20]{Forghani-Kaimanovich15p}
Behrang Forghani and Vadim~A. Kaimanovich, \emph{Boundary preserving
  transformations of random walks}, preprint, 2020.

\bibitem[F{\o}l55]{Folner55}
Erling F{\o}lner, \emph{On groups with full {B}anach mean value}, Math. Scand.
  \textbf{3} (1955), 243--254. \MR{0079220 (18,51f)}

\bibitem[Fur63a]{Furstenberg63a}
Harry Furstenberg, \emph{Noncommuting random products}, Trans. Amer. Math. Soc.
  \textbf{108} (1963), 377--428. \MR{29 \#648}

\bibitem[Fur63b]{Furstenberg63}
\bysame, \emph{A {P}oisson formula for semi-simple {L}ie groups}, Ann. of Math.
  (2) \textbf{77} (1963), 335--386. \MR{26 \#3820}

\bibitem[Fur73]{Furstenberg73}
\bysame, \emph{Boundary theory and stochastic processes on homogeneous spaces},
  Harmonic analysis on homogeneous spaces (Proc. Sympos. Pure Math., Vol. XXVI,
  Williams Coll., Williamstown, Mass., 1972), Amer. Math. Soc., Providence,
  R.I., 1973, pp.~193--229. \MR{50 \#4815}

\bibitem[GdlH17]{Grigorchuk-delaHarpe17}
Rostislav Grigorchuk and Pierre de~la Harpe, \emph{Amenability and ergodic
  properties of topological groups: from {B}ogolyubov onwards}, Groups, graphs
  and random walks, London Math. Soc. Lecture Note Ser., vol. 436, Cambridge
  Univ. Press, Cambridge, 2017, pp.~215--249. \MR{3644011}

\bibitem[Gou15]{Gournay15}
Antoine Gournay, \emph{Amenability criteria and critical probabilities in
  percolation}, Expo. Math. \textbf{33} (2015), no.~1, 108--115. \MR{3310932}

\bibitem[Gre69]{Greenleaf69}
Frederick~P. Greenleaf, \emph{Invariant means on topological groups and their
  applications}, Van Nostrand Mathematical Studies, No. 16, Van Nostrand
  Reinhold Co., New York, 1969. \MR{40 \#4776}

\bibitem[Hul66]{Hulanicki66}
A.~Hulanicki, \emph{Means and {F}\o lner condition on locally compact groups},
  Studia Math. \textbf{27} (1966), 87--104. \MR{0195982}

\bibitem[HY00]{Hayashi-Yamagami00}
Tomohiro Hayashi and Shigeru Yamagami, \emph{Amenable tensor categories and
  their realizations as {AFD} bimodules}, J. Funct. Anal. \textbf{172} (2000),
  no.~1, 19--75. \MR{1749868}

\bibitem[JN15]{Juschenko-Nagnibeda15}
Kate Juschenko and Tatiana Nagnibeda, \emph{Small spectral radius and
  percolation constants on non-amenable {C}ayley graphs}, Proc. Amer. Math.
  Soc. \textbf{143} (2015), no.~4, 1449--1458. \MR{3314060}

\bibitem[JZ18]{Juschenko-Zheng18}
Kate Juschenko and Tianyi Zheng, \emph{Infinitely supported {L}iouville
  measures of {S}chreier graphs}, Groups Geom. Dyn. \textbf{12} (2018), no.~3,
  911--918. \MR{3844999}

\bibitem[Kai90]{Kaimanovich90a}
Vadim~A. Kaimanovich, \emph{Boundary and entropy of random walks in random
  environment}, Probability theory and mathematical statistics, {V}ol.\ {I}
  ({V}ilnius, 1989), ``Mokslas'', Vilnius, 1990, pp.~573--579. \MR{1153846
  (93k:60174)}

\bibitem[Kai92a]{Kaimanovich92b}
\bysame, \emph{Dirichlet norms, capacities and generalized isoperimetric
  inequalities for {M}arkov operators}, Potential Anal. \textbf{1} (1992),
  no.~1, 61--82. \MR{1245225}

\bibitem[Kai92b]{Kaimanovich92}
\bysame, \emph{Measure-theoretic boundaries of {M}arkov chains, {$0$}-{$2$}
  laws and entropy}, Harmonic analysis and discrete potential theory (Frascati,
  1991), Plenum, New York, 1992, pp.~145--180. \MR{94h:60099}

\bibitem[Kai02]{Kaimanovich02a}
\bysame, \emph{The {P}oisson boundary of amenable extensions}, Monatsh. Math.
  \textbf{136} (2002), no.~1, 9--15. \MR{2003e:60013}

\bibitem[Kai05]{Kaimanovich05a}
\bysame, \emph{Amenability and the {L}iouville property}, Israel J. Math.
  \textbf{149} (2005), 45--85, Probability in mathematics. \MR{2191210
  (2007c:43001)}

\bibitem[KF98]{Kaimanovich-Fisher98}
Vadim~A. Kaimanovich and Albert Fisher, \emph{A {P}oisson formula for harmonic
  projections}, Ann. Inst. H. Poincar\'e Probab. Statist. \textbf{34} (1998),
  no.~2, 209--216. \MR{99d:60087}

\bibitem[KV83]{Kaimanovich-Vershik83}
V.~A. Kaimanovich and A.~M. Vershik, \emph{Random walks on discrete groups:
  boundary and entropy}, Ann. Probab. \textbf{11} (1983), no.~3, 457--490.
  \MR{85d:60024}

\bibitem[LS84]{Lyons-Sullivan84}
Terry Lyons and Dennis Sullivan, \emph{Function theory, random paths and
  covering spaces}, J. Differential Geom. \textbf{19} (1984), no.~2, 299--323.
  \MR{86b:58130}

\bibitem[Rei52]{Reiter52}
H.~Reiter, \emph{Investigations in harmonic analysis}, Trans. Amer. Math. Soc.
  \textbf{73} (1952), 401--427. \MR{0051341}

\bibitem[Rei60]{Reiter60}
\bysame, \emph{The convex hull of translates of a function in {$L^{1}$}}, J.
  London Math. Soc. \textbf{35} (1960), 5--16. \MR{0110763}

\bibitem[Rei65]{Reiter65}
\bysame, \emph{On some properties of locally compact groups}, Nederl. Akad.
  Wetensch. Proc. Ser. A 68=Indag. Math. \textbf{27} (1965), 697--701.
  \MR{0194908 (33 \#3114)}

\bibitem[Ren80]{Renault80}
Jean Renault, \emph{A groupoid approach to {$C\sp{\ast} $}-algebras}, Lecture
  Notes in Mathematics, vol. 793, Springer, Berlin, 1980. \MR{82h:46075}

\bibitem[Ric67]{Rickert67}
Neil~W. Rickert, \emph{Amenable groups and groups with the fixed point
  property}, Trans. Amer. Math. Soc. \textbf{127} (1967), 221--232. \MR{222208}

\bibitem[Ros81]{Rosenblatt81}
Joseph Rosenblatt, \emph{Ergodic and mixing random walks on locally compact
  groups}, Math. Ann. \textbf{257} (1981), no.~1, 31--42. \MR{83f:43002}

\bibitem[ST18]{Schneider-Thom18}
Friedrich~Martin Schneider and Andreas Thom, \emph{On {F}\o lner sets in
  topological groups}, Compos. Math. \textbf{154} (2018), no.~7, 1333--1361.
  \MR{3809992}

\bibitem[ST19]{Schneider-Thom19p}
\bysame, \emph{The {L}iouville property and random walks on topological
  groups}, arXiv:1902.10243, 2019.

\bibitem[Ver78]{Vershik78}
A.~M. Vershik, \emph{The action of {${\rm PSL}(2, {\bf Z})$} in {${\bf
  R}\sp{1}$} is approximable}, Uspehi Mat. Nauk \textbf{33} (1978), no.~1(199),
  209--210. \MR{58 \#13201}

\bibitem[VK79]{Vershik-Kaimanovich79}
A.M. Vershik and V.~A Kaimanovich, \emph{Random walks on groups: boundary,
  entropy, uniform distribution}, Dokl. Akad. Nauk SSSR \textbf{249} (1979),
  no.~1, 15--18. \MR{553972 (81f:60098)}

\bibitem[vN29]{vonNeumann29}
John von Neumann, \emph{Zur allgemeinen {T}heorie des {M}a{\ss}es}, Fund. Math.
  \textbf{13} (1929), 73--116 and 333, also: \emph{Collected works}, vol.\ I,
  pages 599--643.

\bibitem[Wei96]{Weinstein96}
Alan Weinstein, \emph{Groupoids: unifying internal and external symmetry. {A}
  tour through some examples}, Notices Amer. Math. Soc. \textbf{43} (1996),
  no.~7, 744--752. \MR{1394388}

\bibitem[Zim77]{Zimmer77}
Robert~J. Zimmer, \emph{Hyperfinite factors and amenable ergodic actions},
  Invent. Math. \textbf{41} (1977), no.~1, 23--31. \MR{57 \#10438}

\bibitem[Zim78]{Zimmer78}
\bysame, \emph{Amenable ergodic group actions and an application to {P}oisson
  boundaries of random walks}, J. Functional Analysis \textbf{27} (1978),
  no.~3, 350--372. \MR{57 \#12775}

\end{thebibliography}

\end{document}